\documentclass{article}
\usepackage{amsfonts}
\usepackage{amsmath, amssymb, amsthm, graphicx}
\usepackage[all]{xy}
\newtheorem{thm}{Theorem}

\newtheorem{cor}[thm]{Corollary}
\newtheorem{lem}[thm]{Lemma}
\newtheorem*{lem2}{Lemma~11.5}
\newtheorem{prob}{Problem}

\newtheorem{defn}[thm]{Definition}

\def\RRA{{\sf RRA}}
\def\wRRA{{\sf wRRA}}
\def\set#1{{\{#1\} }}

\def\Sg{{\sf Sg}}
\let\bp=\cdot
\def\con#1{\setbox13\hbox{$#1$}\ifdim\wd13<1em\breve{#1}\else{(#1)}\breve{\ }\fi}
\def\halfthinspace{\relax\ifmmode\mskip.5\thinmuskip\relax\else\kern.8888em\fi}
\let\hts=\halfthinspace
\def\rp{{\hts;\hts}}
\def\id{{1\kern-.08em\raise1.3ex\hbox{\rm,}\kern.08em}}
\def\di{{0\kern-.04em\raise1.3ex\hbox{,}\kern.04em}}
\def\A{\mathfrak{A}}
\def\L{\mathfrak{L}}

\def\F{\mathbb{F}}

\def\Ek{\mathfrak{E}^{\{2,3\}}_{k+1}}
\def\ie{{\it i.e.}}
\def\conv#1{\setbox13\hbox{$#1$}\ifdim\wd13<1.5em{#1}^{-1}\else{(#1)}^{-1}\fi}
\def\<{\left<}
\def\>{\right>}

\def\set#1{ \{ #1 \}}

\def\c#1{{\mathcal #1}}

\def\W{{\mathsf{W}}}
\def\V{{\mathsf{V}}}

\def\rup#1{{\lceil #1 \rceil}}
\usepackage[colorinlistoftodos]{todonotes}

\def\extra#1{}      
\allowdisplaybreaks

\title{There is no finite-variable equational axiomatization of representable
relation algebras over weakly representable relation algebras}
\author{Jeremy F. Alm\\Department of Mathematics\\Illinois College\\Jacksonville, IL 62650\\
\texttt{alm.academic@gmail.com} 
\and Robin Hirsch\\Computer Science Department\\University College London\\
\texttt{r.hirsch@ucl.ac.uk}
\and Roger D. Maddux\\Department of Mathematics\\Iowa State University\\
\texttt{maddux@iastate.edu}}

\begin{document}
\maketitle
\begin{abstract}
We prove that any equational basis that defines \RRA\ over \wRRA\ must
contain infinitely many variables.  The proof uses a construction of
arbitrarily large finite weakly representable but not representable
relation algebras whose ``small'' subalgebras are representable.
\end{abstract}

\section{Introduction}
J\'onsson~\cite{J59} axiomatized the class of lattices isomorphic to
lattices of commuting equivalence relations. The operations of meet
and join in such lattices are intersection and relational composition,
respectively. Adding converse and the identity relation, J\'onsson
axiomatized the class of algebras isomorphic to algebras of binary
relations with intersection, composition, converse, and identity as
their operations.  Applied to relation algebras, J\'onsson's axioms
yield a characterization of the the class of weakly representable
relation algebras,~\ie, the class of relation algebras isomorphic with
respect to $0,\cdot,\id,\con{\,\,},\rp$ to algebras of binary
relations with set-theoretic constants and operators
$\emptyset,\cap,Id,{}^{-1},|$ (see~\cite[Definition~5.14]{HH}).

J\'onsson asked whether his axioms (which were quasi-equations) could
be replaced by equations.  P\'ecsi~\cite{Pecsi09} proved that they
can.  J\'onsson proved there is a relation algebra that is not weakly
representable, and asked whether there are weakly representable
relation algebras that are not representable. Andr\'eka~\cite{And94} 
not only provided such examples, but showed that no finite number of 
first order conditions are enough to insure that a weakly representable
algebra is representable.

Let \RRA\ denote the class of representable relation algebras, and let
\wRRA\ denote the class of weakly representable relation algebras.
Since \wRRA\ is a variety, Andr\'eka's result says that if $\Sigma$ 
is an equational basis that defines \RRA\ over \wRRA, that is,
$\RRA=\wRRA\cap\text{Mod}(\Sigma)$, then $\Sigma$ cannot be finite.
In the present paper, we strengthen Andr\'{e}ka's result in
Theorem~\ref{MainThm}:
\begin{thm}\label{MainThm}
  Suppose $\Sigma$ is a set of equations such that
  $\RRA=\wRRA\cap\mathrm{Mod}(\Sigma)$.  Then the set of variables
  used by equations in $\Sigma$ is infinite.
\end{thm}
This solves a problem from the first author's dissertation \cite{Alm}.
We solve another problem from \cite{Alm} by exhibiting a non-representable 
relation algebra with a weak representation over a finite set.  In 
addition, we reduce the size of the smallest known weakly representable 
but not representable relation algebra from $2^{366}$ (from \cite{And94}) 
to $2^7$ (see Corollary~\ref{small}).

\section{Proof of Main Result}
\begin{defn}\label{1}
A \emph{relation algebra}
$\A=(A,0, 1, \overline{\phantom{x}}, +,\bp, \id, \con{\phantom{x}}, \rp)$ is a
boolean algebra $(A,0, 1,\overline{\phantom{x}}, +,\bp,)$ together with an
associative binary operation $\rp$ having identity element $\id$, \ie,
$x=x\rp\id$, and a unary operation $\con{\phantom{x}}$, satisfying
additivity: $x\rp(y+z)=x\rp y+ x\rp z$, $\con{x+y}=\con{x}+\con{y}$,
 involution laws: $\con{\con{x}}=x$, $\con{x\rp
y}=\con{y}\rp\con{x}$ and the triangle law: $\con{x}\rp\overline{x\rp y}\leq\overline{y}$.
$\A$ is \emph{symmetric} if it satisfies $\con{x}=x$, for all $x\in A$.  $\A$
is \emph{integral} if $\id$ is an atom.  
\end{defn}  If $\A$ is symmetric then $\A$ is commutative (satisfies
$x\rp{y}=y\rp{x}$), by the involution laws. We only deal with finite symmetric (hence commutative) algebras.

\begin{defn}\label{2}  Let $U$ be an equivalence relation over a set $D$.
A \emph{representation} $\theta$ of a relation algebra $\A$  with unit $U$ over base $D$ is an injective map $\theta:\A\to\mathcal{P}(U)$ sending
each $a\in\A$ to $a^\theta\;(\subseteq U)$, the image of
$a$ under $\theta$, that respects all the relation algebra operators
and constants:
\begin{align*}
0^\theta&=\emptyset\\
1^\theta&=U,
\\ \overline{x}^\theta&=U\setminus x^\theta=\{(u,v)\in U:(u,v)\notin x^\theta\},
\\(x+y)^\theta&=x^\theta\cup y^\theta=\{(u,v)\in U:(u,v)\in x^\theta\lor (u,v)\in y^\theta\},
\\(x\bp y)^\theta&=x^\theta\cap y^\theta=\{(u,v)\in U:(u,v)\in x^\theta\land (u,v)\in y^\theta\},
\\(\id)^\theta&=\{(u,u):u\in D\},
\\\con{x}^\theta&=(x^\theta)^{-1}=\{(u,v)\in U:(v,u)\in x^\theta\},
\\(x\rp y)^\theta&=x^\theta|y^\theta=\{(u,v)\in U:\exists w\in D\big((u,w)\in x^\theta\land (w,v)\in y^\theta\big)\},
\end{align*}
A \emph{weak representation}  is defined similarly, but need not respect union or
complementation.
$\RRA$ and $\wRRA$ are the  classes of relation
algebras that have representations and weak representations,
respectively.  As we mentioned earlier, they are both equational varieties.
Given a representation (or a weak representation) $\theta$ over base $D$,  any $x\in D$ and any $a\in\A$ we write
$\theta(x,a)$ for
\[\set{y\in D:(x,y)\in a^\theta}.\]
\end{defn}
If $\theta$ is any  weak representation of $\c A$ whose unit is some equivalence relation $U$ over base $D$ then for any equivalence class $X$ of $U$ the map $\phi:\c A\rightarrow \mathcal{P}(X\times X)$ defined by $a^\phi=a^\theta\cap(X\times X)$ is easily seen to respect $0, 1, \cdot, 1',  \con{\;},; $ (the unit is now $X\times X$, the base is $X$).   Further, if $\theta$ is a representation then $\phi$ also respects $\overline{\phantom{x}}, +$.
If $\c A$ is integral then it is easy to check, for non-zero $x\in\c A$, that $x;1=1;x=1$ and this ensures that $\phi$ is injective.    A representation (or weak representation) over base $D$ where the unit is $D\times D$ is called \emph{square}.  
Since all the relation algebras considered in this paper are integral, if a representation (respectively weak representation) exists then a square (weak) representation also exists. When we refer to a (weak) representation over a set $D$ the unit will be assumed to be $D^2=D\times D$.

\begin{lem}\label{3}
If $\theta$ is a weak square representation of a relation algebra $\A$ over a
set $D$, then $\theta^m$ is a weak square representation of $\A$ over $D^m$,
where, for every $m\geq1$ and every element $x$ of $\A$,
\begin{equation}\label{eq:rho m}
x^{\theta^m}=\{(u,v)\in D^m\times D^m:\forall i<m\big((u_i,v_i)\in
x^\theta\big)\}.
\end{equation}
\end{lem}
\begin{proof}
The following calculations show that $\theta^m$ maps $\id$, $\bp$,
$\rp$, and $\con{\,\,}$ to the identity relation, intersection,
relative product, and converse, respectively, just because $\theta$
does so.
\begin{align*}
(\id)^{\theta^m}&=\{(u,v)\in(D^m)^2:\forall i<m\big((u_i,v_i)\in(\id)^\theta\big)\}
\\&=\{(u,v)\in(D^m)^2:\forall i<m\big(u_i=v_i\big)\}
\\&=\{(u,v)\in(D^m)^2:u=v\},
\\(x\bp y)^{\theta^m}&=\{(u,v)\in(D^m)^2:\forall i<m\big((u_i,v_i)\in(x\bp y)^\theta\big)\}
\\&=\{(u,v)\in(D^m)^2:\forall i<m\big((u_i,v_i)\in x^\theta\bp y^\theta\big)\}
\\&=\{(u,v)\in(D^m)^2:\forall i<m\big((u_i,v_i)\in x^\theta\land(u_i,v_i)\in y^\theta\big)\}
\\&=\{(u,v)\in(D^m)^2:\forall i<m\big((u_i,v_i)\in x^\theta\big)\land\forall i<m\big((u_i,v_i)\in y^\theta\big)\}
\\&=x^{\theta^m}\cap y^{\theta^m},
\\(x\rp y)^{\theta^m}&=\{(u,v)\in(D^m)^2:\forall i<m\big((u_i,v_i)\in(x\rp y)^\theta\big)\}
\\&=\{(u,v)\in(D^m)^2:\forall i<m\big((u_i,v_i)\in x^\theta\rp y^\theta\big)\}
\\&=\{(u,v)\in(D^m)^2:\forall i<m\,\exists w_i\in(D^m)^2\big((u_i,w_i)\in 
x^\theta\land(w_i,v_i)\in y^\theta\big)\}
\\&=\{(u,v)\in(D^m)^2:\exists w\in(D^m)^2\big((u,w)\in x^{\theta^m}\land(w,v)\in y^{\theta^m}\big)\}
\\&=x^{\theta^m}|y^{\theta^m},
\\\con{x}^{\theta^m}&=\{(u,v)\in(D^m)^2:\forall i<m\big((u_i,v_i)\in\con{x}^\theta\big)\}
\\&=\{(u,v)\in(D^m)^2:\forall i<m\big((u_i,v_i)\in(x^\theta)^{-1}\big)\}
\\&=\{(u,v)\in(D^m)^2:\forall i<m\big((v_i,u_i)\in x^\theta\big)\}
\\&=\{(u,v)\in(D^m)^2:(v,u)\in x^{\theta^m}\}
\\&=(x^{\theta^m})^{-1}.
\end{align*}%
\end{proof}

Roger Lyndon [9] associated a finite algebra $A(G)$ with every finite
projective geometry $G$ of dimension $1$ or more and order $3$ or
more.  $A(G)$ has Boolean algebra whose atoms are the points of $G$
together with a new element $\id$. Every atom (and element) is its own
converse. Relative multiplication is defined only on the atoms and
extended to all of $A(G)$ by additivity.  Lyndon proved [9,~p.23] that
$A(G)$ is a commutative symmetric integral relation algebra.  Lyndon's
proof of associativity explains the need for restricting the order to
$3$ or more (although order $2$ can be accomodated; see [9,~p.24]).
We will now define $\L(p,n)$ by adding new atoms $t_1,\cdots,t_n$
(none if $n=0$) to Lyndon's $A(G)$, where $G$ is the projective
geometry of dimension $1$ and order $p>2$, that is, $G$ is a single
line whose $p+1$ points are $a_0,\cdots,a_p$.  In a Boolean algebra
whose atoms are $\id$, $a_0,\cdots,a_p$, and, if $n>0$,
$t_1,\cdots,t_n$, let the converse of every element be itself, let
$A=a_0+\cdots+a_p$, let $T=t_1+\cdots+t_n$, and define $\rp$ on atoms
as follows: if $0\leq i,j\leq p$, $i\neq j$, $1\leq k,l\leq n$, and
$k\neq l$, then
\begin{align*}
a_i\rp a_i&=\id+a_i,\\
a_i\rp a_j&=A\bp\overline{a_i+a_j},\\
a_i\rp t_k&=T,\\
t_k\rp t_k&=\id+A,\\
t_k\rp t_l&=A.
\end{align*}
Note that $\rp$ is commutative on atoms by its definition, and
commutative on the whole algebra by additivity.
Also, for atoms $q,r,s$, it is easily checked that if $r\leq q\rp s$
then $s\leq q\rp r$.  Suppose the triangle law fails.  Then
$x\rp-(x\rp y)\cdot y\neq0$ for some elements $x,y$, so there are atoms
$q\leq x$ and $r\leq y$ with $q\rp-(x\rp y)\cdot r\neq0$, hence there is an
atom $s\leq-(x\rp y)$ such that $q\rp s\cdot r\neq0$, \ie, $r\leq q\rp s$.  But then
$0\neq s\leq q\rp r\cdot-(x\rp y)\leq x\rp y\cdot-(x\rp y)=0$, a contradiction.
Thus the triangle law holds.

With no new atoms,
$\L(p,0)$ is just Lyndon's relation algebra of the projective geometry
of dimension $1$ and order $p$.  For $n>0$, any product of elements
below $\id+A$ will be the same in both $\L(p,0)$ and $\L(p,n)$.  Since
$\L(p,0)$ is a relation algebra, associativity for $\L(p,n)$ need only
be checked in cases involving the new atoms.  The product (in any
order) of three atoms below $T$ is $T$.  For examples of mixed cases,
consider distinct atoms $t,t'\leq T$ and $a,a'\leq A$.  We have
\begin{align*}
&(a\rp a)\rp t=(\id+a)\rp t=T=a\rp T=a\rp(a\rp t),\\
&(a\rp a')\rp t=(A\bp\overline{a+a'})\rp t=T=a\rp T=a\rp(a'\rp t),\\
&(a\rp t)\rp t=T\rp t=\id+A=a\rp(\id+A)=a\rp(t\rp t),\\
&(a\rp t)\rp t'=T\rp t'=\id+A=a\rp A=a\rp(t\rp t'),
\end{align*}
and the remaining cases follow from these by commutativity.  

Suppose
$0\leq i<j\leq p$, and let
$X=\{\id\}\cup\{t_1,\dots,t_n\}\cup\{a_i+a_j\}\cup
\{a_0,\cdots,a_p\}\setminus\{a_i,a_j\}.$
The product of any two elements of $X$ is a join of elements of $X$,
since
\begin{align*}
  (a_i+a_j)\rp(a_i+a_j) &=\id+A, \\(a_i+a_j)\rp a_k
  &=A\bp\overline{a_k}\text{ for $k\neq i,j$,} \\(a_i+a_j)\rp t_l
  &=T\text{ for $1\leq l\leq n$.}
\end{align*}
Therefore $X$ is the set of atoms of a (maximal) proper subalgebra of
$\L(p,n)$, denoted $\L^{ij}(p,n)$.
\begin{lem}\label{lem:embed}
  If $0\leq i<j\leq p<q$, then $\L^{ij}(p,n)$ is isomorphic to a
  subalgebra of $\L(q,n)$.
\end{lem}
\begin{proof}
  The map from the atoms of $\L^{ij}(p,n)$ to $\L(q,n)$ which maps
  $a_i+a_j$ to $a_i+a_j+a_{p+1}+\cdots+a_q$ and fixes all other atoms
  extends (using additivity) to an embedding of $\L^{ij}(p,n)$ into
  $\L(q,n)$.
\end{proof}

\begin{lem}\label{lem:D}
If $\theta$ is a representation of $\L(p,n)$ over $D$ then
$$p-1=|\theta(x,a_i)|\geq 2n-1$$ for all $x\in D$ and $0\leq i\leq p$.
\end{lem}
\begin{proof} 
Suppose that $\theta$ is a representation of $\L(p,n)$ over $D$. Let
$x\in D$. Then $(x,x)\in(\id)^\theta\subseteq(a_0\rp
a_0)^\theta=a_0^\theta|a_0^\theta$ so there is some $x'\in D$ such
that $(x,x')\in a_0^\theta.$ Now $a_0\leq a_1\rp a_i$ for
$i\in\{2,\dots,p\}$, so there are distinct
$y_2,\dots,y_p\in\theta(x,a_1)$ such that  $(y_i,x')\in a_i^\theta$ for
$i\in\{2,\dots,p\}$, hence $\theta(x,a_1)\supseteq\set{y_2, \ldots, y_p}.$ Conversely, if $x\in \theta(w,a_1)$ then $(w,x')\in
a_1^\theta|a_0^\theta=a_2^\theta\cup\dots\cup a_p^\theta,$ so there is
some $j\in\{2,\dots,p\}$ such that $(w,x')\in a_j^\theta,$ hence
$w=y_j$ because
$$(w,y_j)\in(a_1^\theta|a_1^\theta)\cap(a_j^\theta|a_j^\theta)=(a_1\rp
a_1\bp a_j\rp a_j)^\theta=(\id)^\theta.$$ Therefore $\theta(x, a_1)=\set{y_2, \ldots, y_p}$ and
$|\theta(x,a_1)|=p-1.$ If $n=0$ or $n=1$ then $2n-1\leq p-1$ (since $p\geq 3$) and we are done, so assume
$n\geq2$.

Since $(x,x)\in t_1^\theta|t_1^\theta,$ there is some $x''\in D$ such
that $(x,x'')\in t_1^\theta$, as shown in the diagram below. Since $t_1\leq a_1\rp t_i$ there are distinct
$u_1,\dots,u_n\in\theta(x,a_1)$ such that $(u_i,x'')\in t_i^\theta$,
for $i\in\{1,\dots,n\}$.  Since $t_i\leq a_1\rp t_i$  there are $v_2,\dots,v_n\in D$ such that
$(u_i,v_i)\in a_1^\theta$ and $(v_i,x'')\in t_i^\theta$, for
$i\in\{2,\dots,n\}$.  Note that $v_2,\dots,v_n\in\theta(x,a_1)$ since
$a_1\rp a_1=a_1+\id$, and that $u_1,\dots,u_n,v_2,\dots,v_n$ are
distinct elements of $\theta(x,a_1)$, so $|\theta(x,a_1)|\geq2n-1.$
\[\xymatrix{
&v_i\ar@{-}[rrdd]^{t_i}&&\\
u_i\ar@{-}[ru]^{a_1}\ar@{-}[rrrd]^{t_i}&&&\\
x\ar@{-}[u]^{a_1}\ar@{-}[rrr]^{t_1}\ar@{--}[ruu]_(.6){a_1}&&&x''\\
u_1\ar@{-}[u]^{a_1}\ar@{-}[rrru]_{t_1}&&&
}\]
\end{proof}
\begin{cor}\label{notrap}
If $2n>p$ then $\L(p,n)\notin\RRA$.
\end{cor}
\begin{lem}\label{rep0}
If $p$ is a prime power then $\L(p,0)$ has a representation over a set
of size $p^2$ and $\L(p,1)$ has a representation over a set of size
$2p^2$.
\end{lem}
\begin{proof} The first part was proved in \cite[Theorem~1]{Ly61}, along the following lines.
Let $\F_p$ be the finite field of cardinality $p$. Let $D=\F^2_p$. $D$
is the affine plane with $p$ points on each line. Define some
relations on $D$ as follows. If $0\leq i<p$, $R_i$ is the set of pairs
of distinct points that lie on lines with slope $i$, while $R_p$ is
the set of pairs of distinct points that lie on a ``vertical'' line
(with ``infinite slope'').
\begin{align*}
R_i&=\{(x,y):x,y\in D,\,y-x\in\{(j,ij):0<j\in D\}\}, \text{ for $0\leq i<p$,}
\\R_p&=\{(x,y):x,y\in D,\,y-x\in\{(0,j):0<j\in D\}\}.
\end{align*}
Define a map $\phi:\L(p, 0)\rightarrow \mathcal{P}(D^2)$ by letting $(\id)^\phi$ be the identity over $D$, \/ $a_i^\phi = R_i$ (for $0\leq i\leq p$) and
extend $\phi$ by additivity to arbitrary elements of $\L(p,0)$.  Then $\phi$ is a
representation of $\L(p,0)$ on $D$.  Let ${}':D\to D'$ be a
bijection from $D$ to some disjoint set $D'$ and let $\theta$ be
defined on atoms of $\L(p,1)$ by
\begin{align*}
(\id)^\theta&=\{(x,x):x\in D\cup D'\},
\\a_i^\theta&=R_i\cup\{(x',y'):(x,y)\in R_i\},\text{ for $0\leq i\leq p$,}
\\t_1^\theta&=(D\times D')\cup(D'\times D).
\end{align*}
Extend $\theta$ by additivity to all of $\L(p,1)$.  Then $\theta$ is a
representation of $\L(p,1)$ over $D\cup D'$.
\end{proof}
\begin{cor}\label{cor}
If $p$ is a prime power then $\L(p,0)$ has weak representations over
finite sets of size $p^{2m}$ for all $m\geq1$.
\end{cor}
Let $\theta$ be any weak representation of $\L(p,0)$ over a (possibly
very large) finite base $D$.  Again, let ${}':D\to D'$ be a bijection
from $D$ to some disjoint set $D'$ and let $\theta'$ be the weak
representation of $\L(p,0)$ over $D'$ defined by $(x',y')\in
b^{\theta'}\iff(x,y)\in b^\theta$ (for any $x,y\in D$, $b\in
\L(p,0)$). Next we define a `randomly labelled' $\L(p,n)$ structure
$\xi=\xi(\theta)$ over base $D\cup D'$, as follows.  Partition
$D\times D'$ into $n$ pieces $T_1,\cdots,T_n$ randomly, \ie, each pair
$(x, y')\in D\times D'$ is included in exactly one of the $T_i$ (some
$1\leq i\leq n$) with equal probabilities $\frac 1n$ each, and the
probabilities for distinct edges are independent.  For $b\in \L(p,n)$
let \[b^\xi=(b\cdot (A+\id))^\theta\cup (b\cdot
(A+\id))^{\theta'}\cup\bigcup_{t_i\leq b}(T_i\cup\conv{T_i}).\]

\begin{lem}\label{lem:prob}
Assume $\theta$ is a weak representation of $\L(p,0)$ over a base $D$.
Let $d=|D|$ and $k\leq|\theta(a_i,x)|$ for all $x\in D$, $0\leq i\leq
p$.  Provided
\begin{align}
\label{ineq1}\left(\frac{n^2}{n^2-1}\right)^d&>4n^2d(d-1)\quad\text{and}
\\\label{ineq2}\left(\frac{n}{n-1}\right)^k&>4(p+1)nd^2,
\end{align}
the probability that the random structure $\xi$ is a weak
representation of $\L(p,n)$ is strictly positive.
\end{lem}
\begin{proof}
For any distinct $x,y\in D$, any $z'\in D'$, and any $1\leq i,j\leq
n$ the probability that $(x,z')\in T_i$ and $(y,z')\in T_j$ is
$\frac1{n^2}$.  Hence, for any distinct $x,y\in D$ and any $1\leq i,
j\leq n$ the probability that there is no $z'\in D'$ such that $(x,
z')\in T_i$ and $(y,z')\in T_j$ is $(\frac{n^2-1}{n^2})^d$.  Thus the
probability that there is a distinct pair $x,y\in D$ and some $1\leq
i,j\leq n$ such that there is no $z'\in D'$ witnessing the product
$t_i\rp t_j$ is at most $d(d-1)n^2\left(\frac{n^2-1}{n^2}\right)^d$.
Similarly, for $x\in D$, $y'\in D'$, $0\leq q\leq p$, and $1\leq i\leq
n$, the probability that there is no $z\in D$ such that $(x,z)\in
a_q^\rho$ and $(z,y')\in T_i$ is
$\left(\frac{n-1}n\right)^{|\theta(a_q,x)|}<\left(\frac{n-1}n\right)^k$.
Hence the probability
that $\xi$ fails to be a weak representation is less
than\[2d(d-1)n^2\left(\frac{n^2-1}{n^2}\right)^d+
2(p+1)d^2n\left(\frac{n-1}{n}\right)^k.\]\eqref{ineq1}
and~\eqref{ineq2} ensure that this probability is strictly less than
$\frac12+\frac12$, hence the probability that $\xi$ is a weak
representation is strictly positive.
\end{proof}

\extra{
\begin{lem2}
If $x>(4a)^2$ and $x>e^{\frac b a}$ then $x>a\log_e(x)+b$.
\end{lem2}
\begin{proof}
The first condition is equivalent to $\log_e(x)>2\log_e(4a)$. Recall
that $e^y>y$ for all real $y$.
\begin{align*}
e^{\log_e(x)-\log_e(4a)} &>\log_e(x)-\log_e(4a)\\
\intertext{so}
x&>4a(\log_e(x)-\log_e(4a))\\
&>4a\frac{\log_e(x)}2\;\;(\mbox{by first condition})\\
&=2a\log_e(x)\\
&>a\log_e(x)+b\;\;\;(\mbox{by second condition})
\end{align*}
\end{proof}
}

\begin{thm}\label{weakrep}
If $p\geq3$ is a prime power and $1\leq n$, then $\L(p,n)$ is weakly
representable over  arbitrarily large finite sets.
\end{thm}
\begin{proof}
Let $\theta^m$ be the weak representation of $\L(p,0)$ given in
\eqref{eq:rho m} with base $D=(\mathbb{F}^2_p)^m$, $|D|=p^{2m}$, and
note, for all $x\in D$ and all diversity atoms $a$ of $\L(p,0)$, that
$|\theta(a,x)|=(p-1)^m$.  Observe, in \eqref{ineq1} and
\eqref{ineq2}, that $d=p^{2m}$ and $k=(p-1)^m$ and that the left hand
side of each inequality is governed by a double exponential function
of $m$ whereas the right hand side is governed by only a single
exponential function of $m$.  Hence it is already clear that for
sufficiently large $m$ both inequalities are satisfied.  For such $m$,
there is strictly positive probability that the random structure
$\xi(\theta^m)$ is a weak representation on a base of size $2p^{2m}$ (Lemma~\ref{lem:prob}), hence
a weak representation $\xi$ exists within this probability
space. Routine computation (see the appendix, Lemma~\ref{lem:appendix}) shows that \eqref{ineq1} holds provided
$m>\log_p(16n^2)$ and \eqref{ineq2} holds provided $m>2\log_{p-1}(24n)$ and $m>\frac13\log_{p-1}(4n(p+1))$. 
\end{proof}
For example, by Theorem~\ref{weakrep} and Corollary~\ref{notrap}, we have the  smallest known weakly representable but not representable relation algebra:
\begin{cor}\label{small}
$\L(3,2)$ is a non-representable relation algebra that is weakly 
representable over a finite set.
\end{cor}
\begin{thm}\label{rep}
If $p$ is a prime power and $p$ is large compared to $n\geq 1$, then
$\L(p,n)$ is representable over a finite set of size $2p^2$.
\end{thm}
\begin{proof}
The case $n=1$ is covered by Lemma~\ref{rep0}, so assume $n\geq 2$.
By Lemma~\ref{rep0}, let $\theta$ be a representation of $\L(p,0)$
over a set $D$, where $|D|=p^2$ and $|\theta(u,a_i)|=p-1$.  If $p$ is
sufficiently large compared to $n$ so that \eqref{ineq1} and
\eqref{ineq2} hold then by Lemma~\ref{lem:prob} there is a strictly
positive probability that the random structure $\xi(\theta)$ is a weak
representation, hence a weak representation $\xi$ of this form exists.
Elementary calculations 
show that $p>16n^2$ ensures \eqref{ineq1}
holds and $p>1+(48n)^2$ ensures \eqref{ineq2} holds. 
\extra{
If $p>16n^2$ then $p^2>2n=e^{\frac{\log(4n^2)}{2}}$.  So by Lemma~11.5 with $a=4n^2,\; b=2n^2\log(4n^2)$,
\begin{align*}
p^2>2n^2\log(4n^2)+ 4n^2\log(p^2)&\Leftrightarrow \frac{p^2}{2n^2}>\log(4n^2)+2\log(p^2)\\
&\Rightarrow p^2\log(\frac{n^2}{n^2-1})>\log(4n^2)+2\log(p^2)&&\mbox{by }\eqref{eq:alpha}\\
&\Rightarrow (\frac{n^2}{n^2-1})^{p^2}>4n^2p^2(p^2-1)
\end{align*}
yielding \eqref{ineq1}.
\\
Now suppose $p>(48n)^2+1$.  This is more than enough to get $p-1>e^{\frac{\log(4n)}{6}}=\;^6\!\sqrt{4n}$.  So by Lemma~11.5 with $a=12n,\; b=2n\log(4n)$ we get $(p-1)>2n\log(4n)+12n\log(p-1)$.  Hence
\begin{align*}
\frac{p-1}{2n}>\log(4n)+6\log(p-1)&\Rightarrow (p-1)\log(\frac{n}{n-1})>\log(4n)+4\log(p+1)&&p\geq 16\\
&\Rightarrow (\frac{n}{n-1})^{p-1}>4n(p+1)^5>4n(p+1)^4p
\end{align*}
yielding \eqref{ineq2}.\\
}
Since $\theta$
is a representation (not just a weak one) and since each edge from
$(D\times D')$ is labelled by an atom below $T$, it follows that $\xi$
respects complement and is therefore a representation of $\L(p,n)$.
\end{proof}

\begin{thm}\label{lem:2} 
  For every finite $\gamma$ there exist $p$ and $n$ such that
  $\L(p,n)\in\wRRA\setminus\RRA$ and all the $\gamma$-generated
  subalgebras of $\L(p,n)$ are representable over finite sets.
\end{thm}
\begin{proof}
  Pick any prime power $p$ such that $2^\gamma<p+1$ and pick $n>p/2$.
  Then $\L(p, n)$ is weakly representable by Theorem~\ref{weakrep},
  but not representable by Corollary~\ref{notrap}.  Let
  $\Gamma\subseteq\L(p,n)$ be a set of $\gamma$ generators.  The
  Boolean subalgebra generated by $\Gamma$ (the closure of $\Gamma$
  under intersection and complementation) has at most $2^\gamma$
  atoms. Not all of $a_0,\cdots,a_p$ are among them, because
  $2^\gamma<p+1$. There must be $i<j\leq p$ such that for each
  $g\in\Gamma$ either $a_i+a_j\leq g$ or $(a_i+a_j)\bp g=0$.  This
  implies that $\Gamma$ is a subset of the maximal subalgebra
  $\L^{ij}(p,n)$, because all of its elements are joins of atoms of
  $\L^{ij}(p,n)$.  The subalgebra of $\L(p,n)$ generated by $\Gamma$
  is thus a subalgebra of $\L^{ij}(p,n)$, which is, by
  Lemma~\ref{lem:embed}, (isomorphic to) a subalgebra of $\L(q,n)$ for
  every $q>p$. Choose $q$ so large compared to $n$ that, by
  Theorem~\ref{rep}, $\L(q,n)$ is representable over a finite set.
  Hence the subalgebra of $\L(p,n)$ generated by $\Gamma$ is
  representable over a finite set.
\end{proof}

\begin{proof}[Proof of Theorem \ref{MainThm}]
Suppose $\Sigma$ is a set of equations defining \RRA\ over \wRRA, \ie,
$\RRA=\wRRA\cap\text{Mod}(\Sigma)$.  Also, suppose for contradiction
that there is a finite $\gamma$ such that every equation
$\varepsilon\in\Sigma$ contains only variables from
$x_1,\dots,x_\gamma$.  Choose a large odd prime power $p>2^\gamma-1$ and let
$n=(p+1)/2$.

Since $\L(p,n)$ is not representable, but is weakly representable,
there is some equation $\varepsilon \in\Sigma$ that is not valid in
$\L(p,n)$.  By assumption, $\varepsilon$ contains at most $\gamma$
variables.  Consider an assignment
${}^\prime:\set{x_1,\dots,x_\gamma}\rightarrow\L(p,n)$ to the
variables, falsifying $\varepsilon$.  Let $\Sg(x_1',\dots,x_\gamma')$
be the subalgebra of $\L(p,n)$ generated by
$x_1',\dots,x_\gamma'$. Since each term using only variables
$\set{x_1,\dots,x_\gamma}$ evaluates under ${}^\prime$ to the same
thing in $\L(n,k)$ as in $\Sg(x_1',\dots,x_\gamma')$, this variable
assignment falsifies $\varepsilon$ in $\Sg(x_1',\dots,x_\gamma')$.
But by Theorem~\ref{lem:2}, $\Sg(x_1',\dots,x_\gamma')$ is
representable, yet it fails the equation $\varepsilon\in\Sigma$,
contradicting the assumption $\RRA=\wRRA\cap\text{Mod}(\Sigma)$.
\end{proof}

\section{Equational Complexity}

The following definition of equational complexity from \cite{McNSzeWil08} gives a sort of ``measure'' of non-finite-axiomatizability.

\begin{defn}
The \emph{length} of an equation is the total number of operation
symbols and variables appearing in the equation.  For example, the
length of $(x + y)\cdot z = x\cdot z + y\cdot z$  is 12.   

For a variety $\mathsf{V}$
of finite signature, the \emph{equational complexity} of $\mathsf{V}$ is
defined to be a function $\beta_\mathsf{V}$ where for a positive integer
$m$, $\beta_\mathsf{V}(m)$ is the least integer  such that for any
algebra $\c A$ of the similarity class of $\mathsf{V}$ with $|\c A|\leq m$, $\c A\in
\mathsf{V}$ iff $\c A$ satisfies all equations true in $\mathsf{V}$ of length at most
$\beta_\mathsf{V}(m)$.
More generally, given two varieties $\mathsf{W}\subseteq\mathsf{V}$, the  \emph{equational complexity} of $\W$ over $\V$ is the function  $\beta_{\W/\V}$  where for any positive integer $m$, \/ $\beta_{\W/\V}(m)$ is the least integer such that for any algebra $\c A\in\V$ with $|\c A|\leq m$,\/ $\c A\in\W$ iff $\c A$ satisfies all equations true in $\W$ of length at most $\beta_{\W/\V}(m)$.
\end{defn}
In \cite{Alm12AU}, a log-log lower bound was given for the equational
complexity function for \RRA.  (See also~\cite{McNSzeWil08}.)  Theorem~\ref{MainThm} implies that the equational complexity function of \RRA\ over \wRRA\ must be unbounded; below, we give an explicit lower bound, also log-log. 

\begin{thm}
Let $\beta=\beta_{\RRA/\wRRA}$ be the equational complexity function of \RRA\ over \wRRA.  Then for all $m\geq 2^7$,
\[ \beta(m)> \log_2(2\log_2(m)-5) -\log_23.\]
\end{thm}

\begin{proof} 
From the proof of Theorem~\ref{lem:2}, we have that if $\A$ is a $\gamma$-generated subalgebra of $\L(p,\rup{\frac{p+1}2})$ with $\gamma<\log_2(p+1)$, then $\A$ is representable, hence $\L(p, \rup{\frac{p+1}2})$ satisfies all equations with $\gamma$ variables valid over representable algebras.  Since $\L(p, \rup{\frac{p+1}2})$ is not representable and $|\L(p,\rup{\frac{p+1}2})|=2^{2+p+\rup{\frac{p+1}2}}$, it follows that $\log_2(p+1)\leq\beta(2^{2+p+\rup{\frac{p+1}2}})=\beta(2^{\rup{\frac{3p+5}2}})$.  For any $m\geq 2^7 \;(=2^{\frac{3\times 3+5}2})$ we can find $p\geq 3$ such that $2^{\rup{\frac{3p+5}2}}\leq m< 2^{\rup{\frac{3p+7}2}}\leq 2^{\frac{3p+8}2}$.
Then \begin{equation} \label{eq:p}
 \frac{2\log_2(m)-8}{3}< p.  
\end{equation}
Adding one and then taking logs of both sides of (\ref{eq:p}) yields
\begin{align*}
\log_2\left(\frac{2\log_2(m)-5}{3}\right) &<\log_2(p+1)\\
&\leq\beta(2^{\rup{\frac{3p+5}2}})\\
&\leq \beta(m),
\end{align*}
where the last line follows from monotonicity of $\beta$.  Therefore 
\[
\beta(m) >\log_2(2\log_2(m)-5) - \log_23.
\]

\end{proof}


\section{Open Questions}
Naturally, it seems likely that any equational basis for \wRRA\
contains infinitely many variables.
\begin{prob}
  Does \wRRA\ have a finite-variable equational basis?
\end{prob}
The proof of Lemma~\ref{rep0}, essentially due to Roger Lyndon, shows  that $\L(m, 0)$ is representable whenever there is an affine plane of order $m$.  Furthermore, Lyndon proves the converse: if there is no affine plane of order $m$ then $\L(m, 0)$ is not representable. 

\begin{prob}
Is $\L(m, 0)$ weakly representable for all finite $m\geq 3$?
\end{prob}
If the answer to Problem 2 is ``Yes", that would give a cleaner proof
of the main result of the present paper.  If the answer is ``No", for infinitely many $m$, it would yield a negative answer to Problem 1.

\begin{prob}
Find a reasonable lower bound for the equational complexity
function for \wRRA.
\end{prob}

A Monk algebra is an algebra derived from $\Ek$ by splitting diversity atoms (see \cite{AndMadNem91}).  The algebras $\Ek$ for $1\leq k \leq 400$ were recently shown in \cite{AlmMan} to be representable (except possibly for $k=8,13$). Splitting can destroy representability, however, as in the present paper.  
\begin{prob}
Are all the Monk algebras weakly representable?
\end{prob}

It is known \cite{Jon91} and follows from Theorem~\ref{lem:2} that any equational theory defining \RRA\ must use infinitely many variables, but now consider arbitrary first order theories.
\begin{prob}
Is there a first order theory (necessarily infinite) that defines \RRA\ using only finitely many variables?  If so, how many variables are needed?
\end{prob}
Failing that:
\begin{prob}
Is there a first order theory  (necessarily infinite) that defines \RRA\ over \wRRA\ using only finitely many variables?
\end{prob}

\appendix
\section{Appendix: technical proofs of the inequalities in Lemma~\ref{lem:prob}.}
\begin{lem}\label{lem:appendix} Let $p\geq 3$ be a prime power, let $n\geq 1$ and let $d=p^{2m}, \; k=(p-1)^m$.
\begin{itemize}
\item
If $m>\log_p(16n^2)$ then $\left(\frac{n^2}{n^2-1}\right)^d>4n^2d(d-1)$ (condition \ref{ineq1} of Lemma~\ref{lem:prob}).
\item If $m>2\log_{p-1}(24n),\/ \frac13\log_{p-1}(4n(p+1))$ then $\left(\frac{n}{n-1}\right)^k>4(p+1)nd^2$ (condition \ref{ineq2} of Lemma~\ref{lem:prob}).
\end{itemize}

\end{lem}

\begin{proof}
We claim:
\begin{equation}\label{eq:claim} (x>(4a)^2 \wedge x>e^{\frac b a}) \; \Rightarrow \; x>a\log_e(x)+b.\tag{*}\end{equation}
The  condition $x>(4a)^2$ is equivalent to (**) $\log_e(x)>2\log_e(4a)$. Recall
that $e^y>y$ for all real $y$.  Hence
\begin{align*}
e^{\log_e(x)-\log_e(4a)} &>\log_e(x)-\log_e(4a)\\
\intertext{so, assuming the two conditions on the left hand side of \eqref{eq:claim},}
x&>4a(\log_e(x)-\log_e(4a))\\
&>4a\frac{\log_e(x)}2\;\;(\mbox{by (**)})\\
&=2a\log_e(x)\\
&>a\log_e(x)+b\;\;\;(\mbox{by second condition in \eqref{eq:claim}})
\end{align*}
proving \eqref{eq:claim}.

Note, for $\alpha>1$\begin{equation}\label{eq:alpha}
\log_e\left(\frac\alpha{\alpha-1}\right)>\frac1{2\alpha}.\tag{\dag}\end{equation}
Now for the first part of the Lemma, suppose $m>\log_p(16n^2)$.  Then $p^m>16n^2$ so $p^{2m}>(4\times 4n^2)^2 > 2n=e^{\frac{\log_e(4n^2)}2}$.  So by \eqref{eq:claim}, with $a=4n^2,\; b=2n^2\log_e(4n^2),\; x=p^{2m}$, since $
p^{2m}>(4\times4n^2)^2, e^{\frac{\log_e(4n^2)}{2}}$, we have
\[\begin{array}{crclr}
& p^{2m}&>&4n^2\log_e(p^{2m}) + 2n^2\log_e(4n^2)&\\
\Rightarrow& \frac{p^{2m}}{2n^2}&>&\log_e(4n^2)+2\log_e(p^{2m})&\\
\Rightarrow& p^{2m}\log_e\left(\frac{n^2}{n^2-1}\right) &>& \log_e(4n^2)+2\log_e(p^{2m})&\mbox{by }\eqref{eq:alpha}\\
\Rightarrow& \left(\frac{n^2}{n^2-1}\right)^d&=&
\left(\frac{n^2}{n^2-1}\right)^{p^{2m}}&\\
&&>&4n^2p^{2m}(p^{2m}-1)&\\
&&=&4n^2d(d-1)&
\end{array}
\]
which is Lemma~\ref{lem:prob} condition~\eqref{ineq1}.\\
Now, for the second part of this Lemma, suppose (i) $m>2\log_{p-1}(24n)$ and (ii) $m>\frac13\log_{p-1}(4n(p+1))$. 
 By (ii) we have $(p-1)^m>\;^3\!\sqrt{4n(p+1)}$ and by (i) we have $(p-1)^m>(24n)^2$.  So, by \eqref{eq:claim} with $x=(p-1)^m,\;a=6n, \; b=2n\log_e(4n(p+1))$, we get $(p-1)^m>6n\log_e((p-1)^m)+2n\log_e(4n(p+1))$. Thus
\begin{align*}
\frac{(p-1)^m}{2n}&>\log_e(4n(p+1))+3\log_e((p-1)^m)\\
&\geq\log_e(4n(p+1))+2\log_e(p^m)&&(p\geq 3). 
\end{align*}
Hence (by \eqref{eq:alpha}), we get
\begin{align*}
(p-1)^m\log_e\left(\frac{n}{n-1}\right)&>\log_e(4n(p+1))+\log_e(p^{2m})\\
\intertext{and so}
\left(\frac n {n-1}\right)^k&=\left(\frac{n}{n-1}\right)^{(p-1)^m}\\
&>4(p+1)np^{2m}\\
&=4(p+1)nd^2
\end{align*}
i.e.  Lemma~\ref{lem:prob} condition~\eqref{ineq2} holds.

\end{proof}
\bibliographystyle{plain}

\begin{thebibliography}{10}



\bibitem{Alm}
J.~Alm.
\newblock  Weak representation theory in the calculus of relations.
\newblock ProQuest LLC, Ann Arbor, MI.
\newblock Thesis (Ph.D.),  Iowa State University, 2006.


\bibitem{Alm12AU}
J.~Alm.
\newblock On the equational complexity of RRA.
\newblock {\em Algebra Universalis}, 68(3):321--324, 2012.

\bibitem{AlmMan}
J.~Alm. and J.~Manske.
\newblock Sum-free cyclic multi-bases and
constructions of Ramsey algebras.
\newblock   \emph{Discrete Applied Mathematics}, 180:204--212, 2015. 

\bibitem{And94}
H.~Andr{\'e}ka.
\newblock Weakly representable but not representable relation algebras.
\newblock {\em Algebra Universalis}, 32(1):31--43, 1994.

\bibitem{AndMadNem91}
H.~Andr{\'e}ka, R.~D.~Maddux, and I.~N{\'e}meti.
\newblock Splitting in relation algebras.
\newblock {\em Proc. Amer. Math. Soc.}, 111(4):1085--1093, 1991.


\bibitem{HH}
R.~Hirsch and I.~Hodkinson.
\newblock {\em Relation algebras by games}, volume 147 of {\em Studies in Logic
  and the Foundations of Mathematics}.
\newblock North-Holland Publishing Co., Amsterdam, 2002.




\bibitem{J59}
B.~J\'onsson.
\newblock Representation of modular lattices and of relation algebras.
\newblock {\em Trans. Amer. Math. Soc.}, 92:449--464, 1959.



\bibitem{Jon91}
B.~J\'onsson.
\newblock  The theory of binary relations.
\newblock In {\em Algebraic Logic}, Colloq. Math. Soc. J. Bolyai, 54:245--292, North Holland, 1991.

\bibitem{Ly61}
R.~Lyndon.
\newblock Relation algebras and projective geometries.
\newblock {\em Michigan Mathematics Journal}, 8(1):22--28, 1961.





\bibitem{McNSzeWil08}
G.~McNulty, Z.~Sz{\'e}kely, and R.~Willard.
\newblock Equational complexity of the finite algebra membership problem.
\newblock {\em Internat. J. Algebra Comput.}, 18(8):1283--1319, 2008.

\bibitem{Pecsi09}
B.~P{\'e}csi.
\newblock Weakly representable relation algebras form a variety.
\newblock {\em Algebra Universalis}, 60(4):369--380, 2009.



\end{thebibliography}

\end{document}